\documentclass[11pt]{amsart}
\usepackage{array, amsmath, enumerate, color, amsaddr}
\usepackage{amssymb}
\usepackage{graphicx,subfigure}

\makeatletter

\usepackage{fullpage}

\makeatother

\newtheorem{thm}{Theorem}

\newtheorem{conj}[thm]{Conjecture}

\title{Strong edge-coloring of $2$-degenerate graphs}

\author{Gexin Yu$^{1}$ and Rachel Yu$^{2}$}

\address{
$^1$\small Department of Mathematics, William \& Mary, Williamsburg, VA 23185, USA.\\
$^2$\small Jamestown High School, Williamsburg, VA 23185, USA.
}

\email{gyu@wm.edu}

\begin{document}
\maketitle

\begin{abstract}
A {\em strong edge-coloring} of a graph $G$ is an edge-coloring in which every color class is an induced matching, and the strong chromatic index $\chi_s'(G)$ is the minimum number of colors needed in strong edge-colorings of $G$. A graph is {\em $2$-degenerate} if every subgraph has minimum degree at most $2$. Choi, Kim, Kostochka, and Raspaud (2016) showed $\chi_s'(G) \leq 5\Delta +1$ if $G$ is a $2$-degenerate graph with maximum degree $\Delta$.  In this article, we improve it to $\chi_s'(G)\le 5\Delta-\Delta^{1/2-\epsilon}+2$ when $\Delta\ge 4^{1/(2\epsilon)}$ for any $0<\epsilon\le 1/2$. 
\end{abstract}

\section{Introduction}

A {\em strong edge-coloring} of a graph $G$ is an edge-coloring in which every color class is an induced matching; that is, there are no 2-edge-colored triangles or paths of three edges. The {\em strong chromatic index $\chi_s'(G)$} is the minimum number of colors in a strong edge-coloring of $G$.  This notion was introduced by Fouquet and Jolivet \cite{FJ83} and one of the main open problems was proposed by Erd\H{o}s and Ne\v{s}et\v{r}il~\cite{EN85} during a seminar in Prague: 

\begin{conj}[Erd\H{o}s and Ne\v{s}et\v{r}il, 1985]
If $G$ is a simple graph with maximum degree $\Delta$, then $\chi_s'(G)\le 5\Delta^2/4$ if $\Delta$ is even, and $\chi_s'(G)\le (5\Delta^2-2\Delta+1)/4$ if $\Delta$ is odd. 
\end{conj}

This conjecture is true for $\Delta\le 3$, see \cite{A92, HHT93}. For $\Delta=4$, Huang, Santana and Yu \cite{HSY18} showed that $\chi_s'(G)\le 21$, one more than the conjectured upper bound $20$.  Chung, Gy\'arf\'as, Trotter, and Tuza (1990, \cite{CGTT90}) confirmed the conjecture for $2K_2$-free graphs.  Using probabilistic methods, Molloy and Reed \cite{MR97}, Bruhn and Joos \cite{BJ18}, Bonamy, Perrett, and Postle \cite{BPP22}, and recently Hurley, Verclos, and Kang~\cite{HVK22} showed that $\chi_s'(G)\le 1.772\Delta^2$ for sufficiently large $\Delta$.

Sparse graphs have also attracted a lot of attention.  Interested readers may see the survey paper \cite{DYZ19} for more details. In this article, we are interested in $k$-degenerate graphs when $k=2$.  A graph is {\em $k$-degenerate} if every subgraph has minimum degree at most $k$.  

Let $G$ be a $2$-degenerate graph with maximum degree $\Delta\ge 2$. Chang and Narayanan \cite{CN13} proved that  $\chi_s'(G)\le 10\Delta-10$. Luo and Yu \cite{LY12} improved it to $\chi_s'(G)\le 8 \Delta-4$.  For arbitrary values of $k$,  Wang~\cite{W14} improved the result of Yu~\cite{Y15} and showed the following result.

\begin{thm} \label{thm:k-degenerate}
If $G$ is a $k$-degenerate graph with maximum degree $\Delta\ge k$, then $\chi _s ^{\prime} (G) \leq (4k-2) \Delta - 2k^2 +1$.
\end{thm}

This implies that $\chi_s'(G)\le 6 \Delta-7$ for $2$-degenerate graph $G$. Choi, Kim, Kostochka, and Raspaud~\cite{CKKR18} further improved it to $\chi_s'(G) \leq 5 \Delta +1$ in 2016. Many believe that the optimal bound should be $4\Delta+C$ for some constant $C$, but no progress has yet been made.

In this article, we show that for any $0<\epsilon\le 1/2$, and $\Delta\ge 4^{1/(2\epsilon)}$, $\chi_s'(G)\le 5\Delta-\Delta^{1/2-\epsilon}+2$ for $2$-degenerate graph $G$ with maximum degree $\Delta$.

\section{Main result and its proof}

A {\it special vertex} is a vertex with at most two neighbors of degree more than two.  Every $2$-degenerate graph contains special vertices, which are the new $2^-$-vertices after we remove all vertices of degree at most two.  Let $G$ be a $2$-degenerate graph and $S$ be the set of special vertices of $G$. For each $u\in S$, there exists a set $W_u$ of vertices such that each $w\in W_u$ shares some $2$-neighbors with $u$.  
The {\em capacity of special vertices} of $G$ is the maximum number of common $2$-neighbors that are shared by a vertex in $S$ and other vertices in $G$. A {\em pendant edge} is an edge incident with a leaf (a vertex of degree one).  Below is the main result of this article.

\begin{thm}\label{main}
For any $0<\epsilon\le 1/2$, let $D$ be a positive integer when $D\ge 4^{1/(2\epsilon)}$, for any $2$-degenerate graph $G$ with maximum degree $\Delta$,  if a vertex $u$ is adjacent to at least $d(u)-D$ leaves  when it has $d(u)>D$, and 
\begin{itemize}
    \item $\Delta\le D+2$ when the capacity of special vertices is at least $D^{1/2-\epsilon}$, and
    \item $\Delta\le D+D^{1/2-\epsilon}$ when the capacity of special vertices is less than $D^{1/2-\epsilon}$,
\end{itemize}
then $\chi_s'(G)\le 5D-D^{1/2-\epsilon}+2$. 
\end{thm}

\begin{proof}
Let $G$ be a counterexample with the fewest number of vertices of degree at least two. 

Since $G$ is $2$-degenerate, $G$ must have a set $S$ of special vertices. Let $u\in S$ and $W_u$ be the set whose vertices share $2$-neighbors with $u$ such that the maximum number of common $2$-neighbors of $u$ and vertices of $W_u$ is the capacity of special vertices of $G$. Let $W_u=\{w_1, \ldots, w_s\}$ and $u_1,u_2$ be the two neighbors of $u$ with degree more than $2$. For each $w_i\in W_u$, let $W_i=\{v_{i,1},\ldots,v_{i,t_i}\}$ be the common $2$-neighbors of $w_i$ and $u$. Then $N(u)=\{u_1, u_2\}\cup \bigcup_{i=1}^s \{v_{i,1},\ldots,v_{i,t_i}\} $.  We assume that $t_1\ge t_2\ge \ldots\ge t_s\ge 1$.   Then $t_1$ is the capacity of special vertices of $G$.  It is not hard to see that $u$ has no neighbors of degree one and $w_i$ has degree at least two for each $i$.

For edge $uv$, let $N_2(uv)$ be the set of edges $xy$ such that $x$ or $y$ is adjacent to $u$ or $v$. By definition, $uv$ should have a color different from the colors on edges in $N_2(uv)$ in a valid strong edge-coloring.

\medskip

{\bf Case 1.}  $t_1<D^{1/2-\epsilon}$.  In this case, we have $\Delta\le D+D^{1/2-\epsilon}$.

\medskip

Let $G'$ be the graph obtained from $G-\{uv_{1,1},\ldots,uv_{s,t_s}\}$ by adding up to $D^{\frac{1}{2}-\epsilon}$ pendant neighbors to each of  $\{w_1,\ldots,w_s\}$ so that $w_i$ has degree at most $D+D^{1/2-\epsilon}$ and $w_i$ has at least $D^{\frac{1}{2}-\epsilon}$ pendant neighbors. 
Then the graph $G'$ has fewer vertices of degree at least $2$ and can be colored with $5D-D^{1/2-\epsilon}+2$ colors. We modify the coloring of $G'$ to obtain a coloring of $G$ according to the following algorithm.

\begin{enumerate}
\item Keep the colors of edges that appear in both $G$ and $G'$, but if $w_iv_{i,j}$ for some $i,j$ in $G$ has the same color as $uu_1$ or $uu_2$, then we switch color of $w_iv_{i,j}$ with a color on other pendant edges incident to $w_i$ in $G'$.

\item For each $i$, if a color $c$ appears on both a pendant edge incident to $w_i$ in $G'$ and an edge incident to $u_1$ or $u_2$ (not including $uu_1$ and $uu_2$), then we switch the color of $w_iv_{i,j}$ for some $j$ with the color $c$.

\item After (2), if a color $c$ appears on pendant edges of two or more vertices in $W_u$ in $G'$, then we switch the color of $w_iv_{i,j}$ for some $j$ with $c$ for each such vertex $w_i\in W_u$.

\item After (2) and (3), we color the edges $uv_{1,1},\ldots,uv_{s,t_s}$ in reverse order with colors available to them. 
\end{enumerate}

Now we show that the above algorithm gives a valid strong edge-coloring of $G$.  To do that, we only need to show that each of the edges in $\{uv_{1,1},\ldots,uv_{s,t_s}\}$ can be colored. Consider $uv_{i,j}$ for $1\le i\le s$ and $1\le j\le t_i$.  It needs to get a color not on edges in $N_2(uv_{i,j})$.  Note that $N_2(uv_{i,j})$ contains the edges incident to $u_1, u_2, w_i$ and the edges incident to the $2$-neighbors of $u$; So the number of colored edges in $N_2(uv_{i,j})$, denoted as $n_2(uv_{i,j})$, is at most $$d(u_1)+d(u_2)+d(w_i)+d(u)-3+d(u)-2-t_i-\sum_{p=1}^{i-1}t_p-(j-1).$$ 

We assume that $n_2(uv_{i,j})\geq 5D-D^{1/2-\epsilon}+2$, for otherwise, $uv_{i,j}$ can be colored. Because of the way the edges being colored, we have some repeated colors on edges incident to the $2$-neighbors of $u$, namely $v_{1,1}w_1, v_{1,2}w_1, \ldots, v_{s,t_s}w_s$. The number of colors on $N_2(uv_{i,j})$ and edges incident to $w_i$ for $i\in [s]$ (with repetition) is $n_2(uv_{i,j})+D^{1/2-\epsilon}\cdot (s-1)$.  Thus  $n_2(uv_{i,j})+D^{\frac{1}{2}-\epsilon}(s-1)-(5D-D^{\frac{1}{2}-\epsilon}+2)$ colors are repeated. As each $w_i$ may allow only one edge (for example, $t_i=1$) whose color is the same as other edges in $N_2(uv_{i,j})$,  at least $\frac{n_2(uv_{i,j})+D^{\frac{1}{2}-\epsilon}(s-1)-(5D-D^{\frac{1}{2}-\epsilon}+2)}{D^{\frac{1}{2}-\epsilon}}$ edges have the same colors as others. Since $t_1<D^{1/2-\epsilon}$ and $s\ge \frac{d(u)-2}{t_1}\ge \frac{d(u)-2}{D^{1/2-\epsilon}}$, the number of different colors in $N_2(uv_{i,j})$ is at most
\begin{align*}
& n_2(uv_{i,j})-\frac{n_2(uv_{i,j})+D^{\frac{1}{2}-\epsilon}(s-1)-(5D-D^{\frac{1}{2}-\epsilon}+2)}{D^{\frac{1}{2}-\epsilon}}\\
\le&n_2(uv_{i,j})(1-\frac{1}{D^{\frac{1}{2}-\epsilon}})-(s-1)+5D^{1/2+\epsilon}-1+\frac{2}{D^{\frac{1}{2}-\epsilon}}\\
\leq & (d(u_1)+d(u_2)+2d(u)-5+d(w_i)-t_1)(1-\frac{1}{D^{\frac{1}{2}-\epsilon}})-\frac{d(u)-2}{D^{1/2-\epsilon}}+5D^{1/2+\epsilon}+\frac{2}{D^{\frac{1}{2}-\epsilon}}\\
\le & 3\Delta(1-\frac{1}{D^{\frac{1}{2}-\epsilon}})+d(u)(2-\frac{3}{D^{\frac{1}{2}-\epsilon}})+5D^{1/2+\epsilon}-5+\frac{9}{D^{\frac{1}{2}-\epsilon}}\\
\le & 3(D+D^{\frac{1}{2}-\epsilon})(1-\frac{1}{D^{\frac{1}{2}-\epsilon}})+D(2-\frac{3}{D^{\frac{1}{2}-\epsilon}})+5D^{1/2+\epsilon}-5+\frac{9}{D^{\frac{1}{2}-\epsilon}}\\
\le & 5D+3D^{\frac{1}{2}-\epsilon}-D^{1/2+\epsilon}-8+\frac{9}{D^{\frac{1}{2}-\epsilon}}\le 5D-D^{1/2-\epsilon}+1.
 \end{align*}

The last inequality holds because $D^{1/2+\epsilon}\ge 4D^{1/2-\epsilon}$ and $\frac{9}{D^{\frac{1}{2}-\epsilon}}\le 9$ when $D\ge 4^{\frac{1}{2\epsilon}}$.  Therefore, the edge $uv_{i,j}$ can be colored, which implies that we can color all the uncolored edges. 

\medskip

{\bf Case 2.}  $t_1\ge D^{1/2-\epsilon}$. In this case, $\Delta\le D+2$. 

\medskip

Let $G'$ be the graph obtained from $G$ by deleting the edge $uv_{1,1}$ and adding up to two pendant edges notated $\{w_1v'$, $w_1v''\}$ so that $w_1$ has at least three pendant edges and continues to have maximum degree at most $D+2$. 

Note that for $1\le i\le s$ and $1\le j\le t_s$, 
\begin{equation}
|N_2(uv_{i,j})|\leq d(u_1)+d(u_2)+d(w_i)+d(u)-3+d(u)-2-t_i\le 3\Delta+2D-5-t_i\le 5D+1-t_i.
\label{eq1}
\end{equation}

We observe that $G'$ is 2-degenerate,  has fewer vertices of degree at least $2$, and contains at least $d(v)-D$ vertices of degree one if $d(v)>D$. So, $G'$ can be colored with $5D-D^{1/2-\epsilon}+2$ colors. Keep the colors of the edges in both $G$ and $G'$, we try to get a valid coloring of $G$. 

{\bf Case 2.1.} the color of edge $v_{1,1}w_1$ is different from ones on $\{uu_1, uu_2, uv_{1,2}, \ldots , uv_{s,t_s}\}$.  By \eqref{eq1}, we have at most $5D+1-D^{1/2-\epsilon}$ edges in $N_2(uv_{1,1})$. Since there are $5D-D^{1/2-\epsilon}+2$  different colors available, we can color $uv_{1,1}$ with a color not on the edges in $N_2(uv_{1,1})$.

{\bf Case 2.2.}  the color of edge $v_{1,1}w_1$ is the same as the color of edge $uu_1$ (or edge $uu_2$).  In this case, we swap the color of $v_{1,1}w_1$ with the color of $w_1v'$ or $w_1v''$.  By \eqref{eq1} again, $N_2(uv_{1,1})$ has at most $5D+1-D^{1/2-\epsilon}$ edges in $N_2(uv_{1,1})$. Since there are $5D-D^{1/2-\epsilon}+2$  different colors available, we can color $uv_{1,1}$ with a color not on the edges in $N_2(uv_{1,1})$.

{\bf Case 2.3.} the color of edge $v_{1,1}w_1$ is the same as the color of edge $uv_{i,j}$ for some $2\leq i\leq s$.  In this case, we uncolor the edges $uv_{i,j}, uv_{1,1}, \ldots, uv_{1,t_1}$, and recolor them in the order. By \eqref{eq1}, there are at most $5D+1-t_i-(t_1-1)\le 5D-D^{1/2-\epsilon}+1$ colors on edges in $N_2(uv_{i,j})$, so $uv_{i,j}$ can be colored. Similarly, for each $j$, there are at most $5D-D^{1/2-\epsilon}+1$ colors on edges in $N_2(uv_{1,j})$, so $uv_{1,j}$ can be colored as well.

In any case, $G$ can be colored with $5D-D^{1/2-\epsilon}+2$ colors. So $\chi_s'(G)\le 5D-D^{1/2-\epsilon}+2$.\end{proof}

Observe that our main result follows from Theorem~\ref{main} with $\Delta=D$.  Our proof does not work for corresponding result of list version.

\end{document}